\documentclass[dvips,preprint,imslayout]{imsart} 

\RequirePackage[OT1]{fontenc}
\RequirePackage{amsthm,amsmath,amssymb,natbib}
\RequirePackage[colorlinks,citecolor=blue,urlcolor=blue]{hyperref}
\RequirePackage{hypernat}
\RequirePackage{psfrag}
\RequirePackage[final]{graphicx}
\RequirePackage{pb-diagram,pb-xy}
\RequirePackage[cmtip,arrow]{xy}

\startlocaldefs
 \newcommand{\PP}{\mathbb P}
\newcommand{\R}{\mathbb R}
\newcommand{\EE}{\mathbb E}
\newcommand{\NN}{\mathbb N}

\theoremstyle{plain}

\newtheorem{theorem}{Theorem}

\newtheorem{lemma}[theorem]{Lemma}
\newtheorem{corollary}[theorem]{Corollary}

\theoremstyle{definition}
\newtheorem{definition}{Definition}

\theoremstyle{remark}
\newtheorem{example}{Example}
\newtheorem{remark}{Remark}

\endlocaldefs

\begin{document}

\begin{frontmatter}

\title{Approximation of target problems in Blackwell spaces}
\runtitle{Target problems in Blackwell spaces}


\begin{aug}
\author{\fnms{Giacomo} \snm{Aletti}\thanksref{m1}\corref{}\ead[label=e1]{giacomo.aletti@unimi.it}
\ead[label=u1,url]{http://www.mat.unimi.it/~aletti}}
\author{\fnms{Diane} \snm{Saada}\thanksref{m2}\ead[label=e2]{dianes@mscc.huji.ac.il}}
\affiliation{University of Milan\thanksmark{m1} and
The Hebrew University of Jerusalem\thanksmark{m2}}
\runauthor{G.~Aletti, D.~Saada}

\address{Giacomo Aletti\\Dipartimento di Matematica \\ Via Saldini, 50 \\ 20133, Milano (MI),
ITALY\\ \printead{e1}\\ \printead{u1}}
\address{Diane Saada\\The Hebrew University of Jerusalem, \\Jerusalem, ISRAEL
\\\printead{e2}}
\end{aug}

\begin{abstract}
On a weakly Blackwell space we show how to define a Markov
chain approximating problem, for the target problem. The
approximating problem is proved to converge to the optimal reduced
problem under different pseudometrics. A
computational example of compression of information is discussed.
\end{abstract}



\end{frontmatter}

Let $X_n$ be an homogeneous Markov chain. Suppose the process stops once it reaches an absorbing class, called the target,
according to a given stopping rule: the resulting problem is called
target problem (TP). The idea is to reduce the available information
in order to only use the necessary information which is relevant
with respect to the target. A new Markov chain, associated with a
new equivalent but reduced matrix is defined. In the (large) finite
case, the problem has been solved for TPs: in \cite{A06,A06b,GAEM},
it has been proved that any TP on a finite set of states has its
``best target'' equivalent Markov chain. Moreover, this chain is
unique and there exists a polynomial time algorithm to reach this
optimum.

The question is now to find, in generality, an
$\epsilon$--approximation of the Markov problem when the state space
is measurable. The idea is to merge into one group the points that
$\epsilon$-behave the same with respect to the objective, but also
in order to keep an almost equivalent Markov chain, with respect to
the other ``groups''. The construction of these groups is done
through equivalence relations and hence each group corresponds to a
class of equivalence. In fact, there are many other mathematical
fields where approximation problems are faced by equivalences. For
instance, in integration theory, we use simple functions, in
functional analysis, we use the density of countable generated
subspaces and in numerical analysis, we use the finite elements
method.

In this paper, the approximation is made by means
 of discrete equivalences, which will be defined
in the following. The purpose of any approximation is to reach the
exact solution when $\epsilon \rightarrow 0$. We prove that the
sequence of approximations tends to the optimal exact equivalence
relation defined in \cite{A06,A06b,GAEM}, when we refine the groups.
Finer equivalence will imply better approximation, and accordingly
the limit will be defined as a countably generated equivalence.

Under a very general Blackwell type hypothesis on the measurable space,
we show that it is equivalent to speak on countably generated equivalence
relationships or on measurable real functions on the measurable space of states.
If we do not work under this framework of
Blackwell spaces, we can be faced to paradoxes, as it is explained
by \cite{Dubra-Echenique-2004}, of enlarging $\sigma$--algebras, while decreasing the
information available to a decision-maker.
The $\epsilon$--approximation of the Markov chain depends always upon
the kind of objective. In \cite{J06}, Jerrum deals with ergodic Markov chains. 
His objective is to approximate the stationary distribution by means of a discrete 
approximating Markov chain, whose limit distribution is close in a certain 
sense to the original one. However, unlike our following work, his purpose 
is not the explicit and unified construction of the approximating process.
In this paper, we focus on the target problem.
We solve extensively the TP, where the objective is connected
with the conditional probability of reaching the target $T$, namely
$\PP(X_n \in T \vert X_0=x)$, for any $n,x$.
 This part extends the work in
\cite{A06,A06b,GAEM}, since TPs' approximation may help to
understand the behavior of those TPs where the best equivalent
Markov chain is also very large. The setting of an approximating problem 
can be extended to a general form, but we will not develop it in this paper.
\section{Main results}
Let $(X,\mathcal X)$ be a measurable space. We equip it with an
Assumption~\eqref{eq:A0def} that will be explained when required.
Let $P$ be any
transition probability on $(X,\mathcal X)$. An homogeneous Markov process $(X_n)_n$ is naturally associated to
$(X,\mathcal X,P)$. In the
target problem, we are interested in the probabilities of reaching
the target class $T$ within $n$ steps, namely in
\[
\PP\big(\{X_n\in T\} \big\vert X_0=x\big)
\qquad \text{for any $n$ and $x$}.
\]
The set $T$ is a priori
given and does not change through the computations.
\begin{definition}
Let $(X,\mathcal X)$ be a measurable space and let $T\in{\mathcal X}$.
Let ${\mathcal F}\subseteq{\mathcal X}$ be a
sub $\sigma$-algebra of $\mathcal X$ such that $T\in {\mathcal F}$.
A function $P :X \times {\mathcal F}\rightarrow [0,1]$ is a
\emph{transition probability on $(X,\mathcal F)$} if
 \begin{itemize}
 \item $P(x,\cdot)$ is a probability measure on ${\mathcal F}$, for any
$x\in X$;
 \item $P(\cdot,F)$ is ${\mathcal F}$-measurable, for any
 $F\in{\mathcal F} $.
 \end{itemize}
Given a transition probability $P$ on $(X,\mathcal F)$,
we denote by $P^n$ the transition probability on $(X,\mathcal F)$
given inductively by
\[
P^1 = P; \qquad P^{n+1}(x,F) = \int_X P(x,dy)P^{n}(y,F)
\]
We denote by $\mathrm{TrP}(X,\mathcal X,\mathcal F)$ the set of the
transition probabilities on $(X,\mathcal F)$.
We denote by
${\mathbb{TP}}_X= \cup_{\mathcal F \subseteq \mathcal X}
\mathrm{TrP}(X,\mathcal X,\mathcal F)$ the set of all
transition probabilities on $X$, that we equip with a
suitable pseudometric $d$

\begin{equation*}
d(P_1,P_2)=\sup_x\sum_n\beta^n \big\vert P_1^n(x, T) - P_2^n(x,
T)\big\vert.
\end{equation*}
It is such that
\[
d(P_1,P_2) = 0 \qquad \Longleftrightarrow \qquad P_1^n(x,T) =
P_2^n(x,T), \quad\text{for any $n$ and $x$}.
\]
This pseudo-metric is obviously compatible with the target $T$ and
appears as a loss function of using $P_2$ instead of $P_1$, whatever
is the initial point $x$ ($\beta\in(0,1)$ is a discount rate). The point
will be, with no surprise, to let this distance go to $0$.
\end{definition}
A target problem is defined through a transition probability
$P\in({\mathbb{TP}}_X,d)$.
\begin{definition}
A \emph{target problem} is a quadruple $(X,{\mathcal F},T,P)$,
where $P\in \mathrm{TrP}(X,\mathcal X,\mathcal F)$ and $T\in\mathcal F$.
A \emph{simple target problem} is a target problem where
${\mathcal F}$ is generated by an at most countable partition of $X$.
\end{definition}
The main purpose of this paper is to approximate any target problem
with a sequence of simple target problems in the spirit of the construction
the Lebesgue integral, where the integral  of
a function $f$ is approximated by the integral of simple functions
$f_n=\sum_i c_i I_{C_i}$. The Lebesgue approximation requires at each step $n\in\NN$
two choices: the choice of the subdivision $(C_i)_i$ and the
choice of the function values $(c_i)_i$ on each subdivision.
\begin{definition}
We call \emph{strategy} $\mathrm{Str}$
a sequence of maps $({\mathrm{Str}}_n)_n$ from the set of the target problems
to the set of the simple target problems.
\end{definition}
In the ``Lebesgue example'' given above, the strategy is related to
the ``objective'' of the problem (the integral) and  the pseudometric
$d(f,f_n)=\int|f-f_n|dx$ is required to go to $0$ as $n$ goes to infinity.
Here also, a strategy is meaningful if $d(P,{\mathrm{Str}}_n(P))$ tends
to $0$ as $n$ goes to infinity.
Moreover, for what concerns applications, given a target problem
$(X,\mathcal X,T,P)$
a good strategy should not need the computation of $P^n$, $n>1$.
The first main result of this paper is the existence of a class of good strategies,
called \emph{target algorithms}.
\begin{theorem}\label{thm:mainta}
For any target problem $(X,\mathcal F,T,P)$ and any target algorithm
$\mathrm{Str}$,
\[
\lim_{n\to\infty}d(P,P_n) = 0,
\]
where
$(X,{\mathcal F}_n,T,P_n)={\mathrm{Str}}_n(X,{\mathcal F},T,P) $.
\end{theorem}
Two questions immediately arise:
does the sequence $({\mathrm{Str}}_n(P))_n$ have a limit (and in which sense)?
Moreover, since $d$ is defined as a pseudometric,
does this limit depend on the choice of $\mathrm{Str}$?

The extension of the concept of compatible projection given in \cite{GAEM,A06,A06b}
to our framework will enable us to understand better the answer to these
questions.
A measurable set $A \neq \emptyset$ of a measurable space $(X,
\mathcal X)$ is an $\mathcal X$-atom if it has no non-empty
measurable proper subset. No two distinct atoms intersect. If the
$\sigma$-field is countably generated, say by the sequence $\{A_n\}$
then the atoms of $\mathcal X$ are of the form $\cap_nC_n$ where
each $C_n$ is either $A_n$ or $X\setminus A_n$.
\begin{definition}
An equivalence relationship $\pi$ on a measurable set $(X,\mathcal X)$ is
\emph{measurable} (\emph{discrete}) if there exists a (discrete)
random variable $f:(X,\mathcal X)\to(\R,\mathcal B_\R)$
($\mathcal B_\R$ denotes the Borel $\sigma$-algebra),
such that
\[
x\,\pi\,y \qquad \Longleftrightarrow \qquad f(x) = f(y),
\]
and we denote it by $\pi=\pi_f$.
Let $(X,{\mathcal F},T,P)$ be a target problem.
A \emph{compatible projection} is a measurable equivalency $\pi_f$ such that
\begin{equation}\label{defn:comp_proj}
P(x,F) = P(y,F),\qquad \forall x\,\pi_f\,y, \forall
F\in\sigma(f).
\end{equation}
A compatible projection $\pi$ is said to
be \emph{optimal} if $\pi\supseteq\pi'$, for any other compatible projection
$\pi'$.
\end{definition}
\begin{remark}
This definition is well posed if
\[\pi_f=\pi_g \Longleftrightarrow \sigma(f)=\sigma(g).\]
Assumption~\eqref{eq:A0def} ensures that the definition of
measurable equivalency is indeed well posed. This assumption will be stated
and discussed in Section \ref{sec:A0}.
\end{remark}
\begin{theorem}\label{thm:CP+TP}
If $\pi=\pi_f$ is a compatible projection for the target problem
$(X,{\mathcal F},T,P)$, then there exists
a target problem $(X,\sigma(f),T,P_\pi)$.
such that $P_\pi(x,F) = P(x,F)$ for any
$F\in\sigma(f)$.
\end{theorem}

It is not said ``a priori'' that an optimal compatible projection
must exist. If it is the case, then this equivalence is
obviously unique.

\begin{theorem}\label{thm:maintb}
For any target problem $(X,\mathcal F,T,P)$, there exists a (unique) optimal
compatible projection $\pi$.
\end{theorem}

To conclude the main results, let us first come back to the Lebesgue example.
The simple
function $f_n=\sum_i c_i I_{C_i}$ are chosen so that $\sigma(C_n)$ increases
to $\sigma(f)$
and $f_n(x) \to f(x)$. The following theorem guarantees these two facts by showing
the ``convergence'' of any strategy to the optimal problem.
\begin{theorem}\label{thm:maintc}
Let ${\mathrm{Str}}_n (X,{\mathcal F},T,P) = (X,{\mathcal F}_n,T,P_n)$,
with $\mathrm{Str}$ target algorithm and let
$\pi$ be the optimal compatible projection associated to the target problem
$(X,{\mathcal F},T,P)$. Then
\begin{itemize}
  \item ${\mathcal F}_n \subseteq {\mathcal F}_{n+1}$ for any $n$, and
  $ \vee_n {\mathcal F}_n = {\mathcal F}_\pi$;
  \item $\lim_n P_n (x,F) = P_\pi(x,F)$, for any $F\in \cup_m {\mathcal F}_m$.
\end{itemize}
\end{theorem}

\begin{remark}[The topology $ \mathrm{Top} $] In
Theorems~\ref{thm:mainta}--\ref{thm:maintc}, we have proved the
convergence of $
({P}_n)_n $ to $P_\pi$ with respect to the pseudometric $d$.
The pseudometric topology $\mathrm{Top} $ is the topology induced
by the open balls $B_r(P)=\{Q\in {\mathbb{TP}}_X\colon d(P,Q)<r \}$,
which form a basis for the topology.
Accordingly, the previous theorems may be reread in
terms of convergence of $P_n$ to
$P$ on the topological space $({\mathbb{TP}}_X,\mathrm{Top})$.
\end{remark}

\subsection{Connection with weak convergence}
Given a strategy $(X,{\mathcal F}_n,T,P_n)_n$,
if we want to show a sort of weak convergence of
$P_n(x,\cdot)$ to $P(x,\cdot)$, for any $x$, we face the two following problems
\begin{itemize}
    \item each $P_n(x,\cdot)$ is defined on a different
    domain (namely, on $\mathcal F_n$);
    \item we did not have required a topology on $X$.
\end{itemize}
First, we want to introduce a new definition of probability convergence
which takes into account the first restriction. The idea is given in
the following example.
\begin{example}\label{exa:dyadic}
Let $\mathcal F_n =
\sigma(\{(i2^{-n},(i+1)2^{-n}],i=0,\ldots,2^n-1\})$ be the
$\sigma$-algebra on $(0,1]$ generated by the dyadic subdivision.
Suppose we know that $\nu_n:\mathcal F_n\to [0,1]$ is the unique
probability on $\mathcal F_n$ s.t.\ for any $i$,
$\nu_n((i2^{-n},(i+1)2^{-n}])=2^{-n}$. Even if $\nu_n$ is not
defined on the Borel sets of $(0,1]$, it is clear that in ``some''
sense, it must happen that $\nu_n\to\nu_*$, where $\nu_*$ is the
Lebesgue measure on the Borel sets of $(0,1]$. Note that the
cumulative function of ${\nu}_n$ is not defined, and therefore a
standard weak convergence cannot be verified.

In fact, we know that
\begin{equation}\label{eq:exadyadic}
\nu_n \Big( \big( -\infty,  \frac{i}{2^n} \big] \Big) = \nu_n
\Big( \big( 0,  \frac{i}{2^n} \big] \Big) = \frac{i}{2^n},
\end{equation}
i.e., in this case, as $n\to\infty$, we can determine  the
cumulative function in a dense subset. This fact allows to hope that
$\nu_n\to\nu_*$ in a particular sense.
\end{example}

\begin{definition}
Let $(X,\mathcal X,(\mathcal X_n)_n)$ be a filtered space, and set
${\mathcal X}_{\infty}=\vee_n {\mathcal X}_n$. Let $\nu_n:{\mathcal
X}_n \rightarrow [0,1]$, $n\geq 1$ and $\nu_{\infty} :{\mathcal
X}_{\infty} \rightarrow [0,1]$ be probability measures. We say that
$\nu_n$ \emph{converges totally to $\nu_{\infty}$ on the topological
space $(X,\tau)$} as $n$ tends to infinity, if $\bar{\nu}_n
\mathop{\longrightarrow}\limits^{w}_\tau \nu_{\infty}$ (converges in
weak sense on $(X,\tau)$), for any $\bar{\nu}_n :{\mathcal
X}_{\infty} \rightarrow [0,1]$, such that
${\bar{\nu}_n}{}_{\vert_{{\mathcal X}_n}}=\nu_n$. We write $\nu_n
\mathop{\longrightarrow}\limits^{\mathrm{tot}}_\tau \nu_{\infty}$.
\end{definition}

Going back to the example, it is simple to check that $\nu_n
\mathop{\longrightarrow}\limits^{\mathrm{tot}}_{\tau(0,1]} \nu_*$,
where $\nu_n,\nu_*$ are given in Example~\ref{exa:dyadic} and
$\tau(0,1]$ is the standard topology on $(0,1]$. In fact, let
$(\bar{\nu}_n)_n$ be any extension of $({\nu}_n)_n$ to the Borel
sets of $(0,1]$. For any $t\in (0,1)$, we have by
\eqref{eq:exadyadic} that
\[
t-\frac{1}{2^n} \leq F_{\bar{\nu}_n} (t) \leq t+\frac{1}{2^n} ,
\]
where $F_{\bar{\nu}_n}$ is the cumulative function of
$\bar{\nu}_n$, which implies the weak convergence of $\bar{\nu}_n$
to $\nu_*$ and, therefore, ${\nu}_n
\mathop{\longrightarrow}\limits^{\mathrm{tot}}_{\tau(0,1]} \nu_*$.

For what concerns the topology on $X$, we will define the topological space  $(X,\varrho_P)$
induced by the pseudometric $d_P$ associated to the target problem
$(X,{\mathcal F},T,P)$, and the pseudometric $d$. In this way
$\varrho_P$ is defined only with the data of the problem.
One may ask: is this topology too poor?
The answer is no, since it is defined by
the interesting pseudometric $d_P$.
In fact, $d_P(x,y)<\epsilon$
 means that $x$ and $y$ play ``almost the same role'' with
respect to $T$. A direct algorithm which takes $d_P$ into
account needs the computation of $P^n$ at each step. In any case, even if $d_P$
may not be computable, it defines a nontrivial interesting topology
$\varrho_P$ on $X$. As expected, we have the following theorem.
\begin{theorem}\label{prop:Top}
Let ${\mathrm{Str}}_n (X,{\mathcal F},T,P) = (X,{\mathcal F}_n,T,P_n)$,
with $\mathrm{Str}$ target algorithm. Then
\[
P_n
\mathop{\longrightarrow}\limits^{\mathrm{tot}}_{\varrho_P} P.
\]
\end{theorem}

\section{The target algorithm}
In this section, we introduce the core of the approximating target
problem, namely a set of strategies $\mathrm{Str}$ which solves the
target problem.


Given a measurable space $(X,\mathcal X)$ and a target
problem $(X,\mathcal F,T,P)$,
the target algorithm is built in the spirit of the exact one given
in \cite{A06,A06b}, which starts from the largest classes $T$ and $X
\setminus T$ and then reaches the optimal classes according to a
backward construction.

The target algorithm defines a strategy $\mathrm{Str} = (\mathrm{Str}_n)_n$,
where
\[
\mathrm{Str}_n(X,\mathcal F,T,P) = (X,\mathcal F_n,T,P_n),
\]
and it consists of three steps:
\begin{enumerate}
  \item the choice of a sequence $(\sim_{\epsilon_n})_n$ of equivalences on
  the simplex on the unit ball of $\ell_1$ with $\epsilon_n\to 0$;
  \item the definition of a filtration $(\mathcal F_n)_n$ based on
  $(\sim_{\epsilon_n})_n$ where each $\mathcal F_n$ is generated by
  a countable partition of $X$;
  \item the choice of a suitable measure $\mu$
  and the definition of $(P_n)_n$.
\end{enumerate}

\subsection{Preliminary results on measurability and equivalency, and
the choice of $(\sim_{\epsilon_n})_n$}
Associated to each
countably generated sub $\sigma$-algebra $\mathcal A\subseteq
\mathcal X$, we define the equivalence relationship $\pi_{\mathcal
A}$ induced by the atoms of $\mathcal A$:
\[
x\,\pi_{\mathcal A}\,y \iff
[x]_{\mathcal A} :=\cap \{A\in\mathcal A\colon x\in A\}=
\cap \{A\in\mathcal A\colon y\in A\}=:[y]_{\mathcal A}.
\]
Thus, if $(\mathcal A_n)_n$ is a sequence of countably
generated $\sigma$-algebras, then
\begin{equation}\label{eq:vee_and_cap}
\pi_{\vee_n\mathcal A_n} = \cap_n \pi_{\mathcal A_n} .
\end{equation}

Now, the atoms of the $\sigma$-algebra $\mathcal F$ of
each simple target problem $(X,\mathcal F,T,Q)$ are at most countable,
by definition. Then $Q$ may be represented as a transition
matrix on the state set $\NN$. Each row of $Q$ is a distribution
probability on $\mathbb N$ (i.e. a sequence $(p_n)_n$ in the simplex $S$ of
$\ell_1$).
The first step of the target algorithm is to equip $S$ with the
$\ell_1$--norm and then to define an $\epsilon$-equivalence on $S$.

\null

We will alternatively use both the discrete equivalencies
and the countable measurable partitions, as a consequence of the
following result, whose proof is left to appendix.
\begin{lemma}\label{lem:discrete_measur}
Given a measurable space $(X,\mathcal X)$,
there exists a natural bijection between the set of discrete equivalencies
on $X$ and the set of the countable measurable partitions of it.
\end{lemma}
Let $B_{\ell_1}(0,1)$ be the unit ball in $\ell_1$
and $S=\{x \geq 0\} \cap B_{\ell_1}(0,1)$
be the simplex on $\ell_1$. Let $\Omega_n=[0,1]$, for any $n$,
and $\tau$ be the standard topology on $[0,1]$. Denote by $\mathcal
B_{[0,1]}$ the Borel $\sigma$-algebra on $[0,1]$ generated by
$\tau$. We look at $S$ as a subset of $\Pi_{n=1}^{\infty}\Omega_n$
so that the Borel $\sigma$-algebra $\mathcal B_S$ induced on $S$ is
$\bigotimes_{n=1}^{\infty}\mathcal B_{[0,1]} \cap S$.
\begin{definition}
$\sim_{\epsilon}$ is an \emph{$\epsilon$-equivalence on $S$} if
it is a discrete equivalence on $(S,\mathcal B_S)$
and $\|p-q\|_1<\epsilon$ whenever $p \sim_{\epsilon} q$.
\end{definition}
\begin{remark}
The choice of $\ell_1$--norm on $S$ is
linked to the total variation distance between probability measures.
The total variation distance  between two probability measures $P$
and $Q$ is defined by $d_{TV}(P,Q)=\sup_{A \in \Omega}\vert
P(A)-Q(A)\vert$. Now the total variation of a measure $\mu$ is $\|
\mu \|(\Omega)=\sup \sum_i \vert \mu(A_i) \vert$, where the supremum
is taken over all the possible partitions of $\Omega$. As
$(P-Q)(\Omega)=0$, we have that $d_{TV}(P,Q)=\frac{1}{2}\| P-Q \|$,
see \cite{Billingsley:book}. To each $p\in S$ corresponds the
probability measure $P$ on $\mathbb N$ with $P({i}) = p_i$
 (and viceversa).
Therefore, since $\|p-q\|_1= \|P-Q\|= 2 d_{TV}(P,Q)$, we have
\[
p \sim_{\epsilon} q \Longrightarrow d_{TV}(P,Q) < \epsilon/2
\]
\end{remark}
\begin{example}\label{exa:e-cut}
Define the $\epsilon$-cut as follows. $p \sim_{\epsilon} q \iff
\left\lfloor \frac{p_n}{\epsilon2^{-n}} \right\rfloor= \left\lfloor
\frac{q_n}{\epsilon2^{-n}} \right\rfloor, \forall n$, where $\lfloor
x\rfloor$ denotes the entire part of $x$. Then
$\sim_{\epsilon}$ is an
$\epsilon$-equivalence on $S$. Indeed,
\begin{itemize}
\item $S/\sim_{\epsilon}$ is at most countable (since we divide
each $[0,1]$ into classes of length $\epsilon2^{-n}$).
\item
For any $ p \in S $
 \[
[p] =\{q \in S\colon \pi_{\sim_{\epsilon}}(q)
=\pi_{\sim_{\epsilon}}(p) \} = \prod_{n} \bigg[ \frac{\Big\lfloor
\frac{2^np_n}{\epsilon} \Big\rfloor\epsilon}{2^{n}} ,
\frac{\Big(\Big\lfloor \frac{2^np_n}{\epsilon}
\Big\rfloor+1\Big)\epsilon}{2^{n}} \bigg) \,
 {{\bigcap S}}
\] is measurable with respect to ${\mathcal B}_S.$
 \item
$\forall p \sim_{\epsilon}q$,
\[
\|p-q\|_1 \leq \sum_n \epsilon2^{-n} = \epsilon .
\]
\end{itemize}
\end{example}

\subsection{The choice of $(\mathcal F_n)_n$}

Given a sequence
$(\sim_{\epsilon_n})_{n \in\mathbb N}$ of $\epsilon$-equivalences on
$S$, we define the choice of $(\mathcal F_n)_n$ inductively.
This algorithm is a good candidate to be a
strategy for the approximating problem we are facing and it is based
on this idea: consider the
equivalence classes given by ${\mathcal F}_{n-1}$ and divide them again
according to the following rule. Starting from any two points in the
same class, we check whether the probabilities to attain any other
${\mathcal F}_{n-1}$-classes are $\epsilon$-the same. Mathematically speaking:

Step $0$ : $\mathcal F_0 = \sigma(T)=\{\varnothing,T, X \setminus T,X\}$

Step $n$ :
${\mathcal F}_n$ is based on the equivalence ${\mathcal F}_{n-1}$ and
on $\sim_{\epsilon_n}$, inductively. ${\mathcal F}_{n-1}$ is
generated by a countable partition of $X$, say $(A^{(n-1)}_i)_i$.
We define, for any couple $(x,y)\in
X^2$,
\begin{equation}\label{eq:projectionsta}
(x \pi_{n}y) \iff (x \pi_{n-1}y)
\wedge \Big(\big(P(x,A^{(n-1)}_i)_i\big)
\sim_{\epsilon_n}\big(P(y,A^{(n-1)}_i)_i\big)\Big)  .
\end{equation}
The following Lemma~\ref{lem:contmeas} shows that $\pi_n$ is
a discrete equivalency on $(X,\mathcal X)$,
and therefore it defines ${\mathcal F}_{n}=\sigma(X/\pi_n)$ as
generated by a countable partitions of $X$.
%
%

\begin{remark}
In applications, $(\mathcal F_n)_n$ must be finitely generated. This
is not a big restriction. In fact one can prove inductively that this
is always the case if the projection  $\sim_{\epsilon_n}$ divides
each component of $S$ into a finite number of subsets, as in
 Example~\ref{exa:e-cut}. The choice of the ``optimal''
sequence $(\sim_{\epsilon_n})_n$ is not the scope of this work.
We only note that the definition of $\sim_{\epsilon}$ can be relaxed
and the choice of the sequence $(\sim_{\epsilon_n})_n$ may be done
interactively, obtaining a fewer number of classes $(A^{(n)}_i)_i$
at each step.
\end{remark}

\begin{lemma}\label{lem:contmeas}
$(\mathcal F_n)_n$ is a filtration on $(X,\mathcal F)$. Moreover,
for any $n\in\NN$, $\pi_n$ is
a discrete equivalency on $(X,\mathcal X)$.
\end{lemma}
\begin{proof}
The
monotonicity of $(\mathcal F_n)_n$ is a simple consequence of
\eqref{eq:projectionsta}.

The statement
is true for $n=0$, since $T\in\mathcal X$. For the induction step,
let $\{A_1^{(n-1)},A_2^{(n-1)},\ldots\}\in\mathcal X$ be the
measurable countable partition of $X$ given by $X/\pi_{n-1}$. The
map $h:(X,\mathcal X)\to (S,\mathcal B(S))$ given by $ x\mapsto
(P(x,A_i^{(n-1)}))_i$ is therefore measurable. As $\sim_{\epsilon_n}$
is a discrete equivalency on $(S,\mathcal B_S)$, the map
$\pi_{\sim_{\epsilon_n}}\circ h:(X,\mathcal X)\to
(S/{\sim_{\epsilon_n}},2^{S/{\sim_{\epsilon_n}}})$ is also measurable,
where  $\pi_{\sim_{\epsilon_n}}$ is the natural projection
associated with $\sim_{\epsilon_n}$.
Thus, two points $x,y\in X$ are such that
\[
\Big(\big(P(x,A^{(n-1)}_i)_i\big)
\sim_{\epsilon_n}\big(P(y,A^{(n-1)}_i)_i\big)\Big)
\]
if and only if their image by $\pi_{\sim_{\epsilon_n}}\circ h$
is the same point of ${S/{\sim_{\epsilon_n}}}$. The new
partition of $X$ built by $\pi_n$ is thus obtained as an intersection of
the sets $A_i^{(n-1)}, i \geq 1$ ---which formed the
$\pi_{n-1}$-partition--- with the counter-images of ${S/{\sim_{\epsilon_n}}}$ by
$\pi_{\sim_{\epsilon_n}}\circ h$. Intersections between two
measurable countable partitions of $X$ being a measurable countable
partition of $X$, we are done.
\end{proof}


\subsection{The choice of $\mu$ and the definition of $(P_n)_n$}
Before defining $(P_n)_n$, we need the following result,
which will be proved in Section~\ref{sec:proofs}.
\begin{theorem}\label{thm:pi_infty}
Let $(\pi_n)_n$ be defined as in the previous section and
let $\pi_\infty = \cap_n \pi_n$. Then $\pi_\infty$ is a compatible
projection.
\end{theorem}
As a consequence of Theorem~\ref{thm:CP+TP} and of Theorem~\ref{thm:pi_infty},
a target problem $(X,\vee_n {\mathcal F}_n,T,P_{\infty})$ is well defined.
We intend to define $P_n$ as the $\mu$--weighted mean average
of $P_{\infty}$ given the information carried by
${\mathcal F}_n$.

More precisely, let $\mu$ be a probability measure on
$(X,\vee_n{\mathcal F}_{n})$ such that $\mu(F)>0$, for
any $F \in {\mathcal F}_n, F\neq\varnothing$ (the existence of
such a measure is shown in Example~\ref{exa:exist_meas}).

For any $F\in{\mathcal F}_{n}$, let $Y^F$ be the
$\vee_n{\mathcal F}_{n}$-random variable
such that $Y^F(\omega)= P_{\infty}(\omega,F)$. Define
\begin{equation}\label{eq:def_of_P_n}
{P}_n(x,F)= \EE_{\mu}[Y^F\vert {\mathcal F}_{n}](x), \qquad
\forall x\in X, \forall F\in {\mathcal F}_n.
\end{equation}
${P}_n$ is uniquely defined on $(X\times {\mathcal F}_n)$,
the only $\mu$-null set of ${\mathcal F}_n$ being the empty set.
Then we can ensure that ${P}_n(x,\cdot)$ is a probability
measure, for any $x\in X$.

\null

We give in the following an example of the measure $\mu$ that has
been used in Equation~\eqref{eq:def_of_P_n} which justifies its
existence.
\begin{example}\label{exa:exist_meas}
Let $(Y_n)_{n \geq 0}$ be a sequence of independent and identically
distributed geometric random variables, with $\mathbb P_{Y_i} ( j )
= 1 / 2^j , j\in\mathbb N$. Let ${\mathcal A}_n=\sigma(Y_0, \cdots,
Y_n)$ and set ${\mathcal A}=\vee{\mathcal A}_n$. There exists a
probability measure $\mathbb{P}$ on ${\mathcal A}$ such that
\begin{equation*}
\mathbb{P}(\cap_{i=0}^n\{Y_{l_i}=y_i\})=\mathbb{P}_{Y_{l_1}}(y_1)
\otimes\cdots \otimes
\mathbb{P}_{Y_{l_n}}(y_n)=\frac{1}{2^{\sum_{i=0}^ny_i}},
\end{equation*}
and thus, $\mathbb{P}(A)>0$, $\forall A\in\mathcal A_n,
A\neq\varnothing$. Moreover, it follows that for any $n$,
\begin{equation}\label{eq:mathbbPonFnYn}
A_1 \in\mathcal A_n , A_2 \in\sigma (Y_{n+1}), A_1\neq
\varnothing, A_2\neq \varnothing, \quad \Longrightarrow \quad
\mathbb{P}(A_1\cap A_2)>0.
\end{equation}
We check by induction that we can embed ${{\mathcal F}}_n $ into
${\mathcal A}_n$, for any $n \geq 0$. The searched measure $\mu$
will be the trace of $\mathbb P$ on the embedded $\sigma$-field
$\vee_n{\mathcal F}_{n}$.

For $n=0$, define $T \mapsto \{Y_0 = 1\}$, $X \setminus T \mapsto
\{ Y_0\geq 2\}$. The embedding forms a nontrivial partition, and
therefore the restriction of $\mathbb P$ to the embedding of $
{\mathcal F}_0$ defines a probability measure on ${\mathcal F}_0$ with
$\mu_0(F)>0$ if $F\neq\varnothing$.

For the induction step, suppose it is true for $n$. Given
$F_i^{(n)}\in{\mathcal F}_n$, we then have \( F_i^{(n)} \mapsto
A_i^{(n)} \), where \((A_i^{(n)})_i \) is a nontrivial partition in
$\mathcal A_n$ and therefore the restriction of $\mathbb P$ to the
embedding of $ {\mathcal F}_n$ defines a probability measure $\mu_n$ on
${\mathcal F}_n$ with $\mu_n(F)>0$ if $F\neq\varnothing$.

Given $F_i^{(n)}$, let $H_i^{(n+1)}:=\{F_j^{(n+1)}\colon F_j^{(n+1)}
\subseteq F_i^{(n)} \}$. The monotonicity of $\pi_n$ ensures that
each $F_j^{(n+1)}$ will belong to one and only one $H_i^{(n+1)}$.
Moreover, by definition of $F_j^{(n+1)}$, we have that
\begin{equation}\label{eq:examu1}
F_i^{(n)} = \cup \{ F_j^{(n+1)}\colon F_j^{(n+1)}\in H_i^{(n+1)}
\}.
\end{equation}
Since $X/\pi_{n+1}$ is at most countable, we may order $H_i^{(n+1)}$
for any $i$. We have accordingly defined an injective map
$X/\pi_{n+1}\to \mathbb N^2$, where
\[
F_j^{(n+1)} \mapsto (i,k) \iff F_j^{(n+1)} \text{ is the $k$-th
element in }H_i^{(n+1)} .
\]

According to the cardinality of $H_i^{(n+1)}$, define the
$n+1$-embedding
\[
F_j^{(n+1)} \mapsto (i,k) \mapsto A_j^{(n+1)} := A_i^{(n)} \cap
\begin{cases}
\{ Y_{n+1} = k \} & \text{ if }k < \#\{H_i^{(n+1)}\}
\\
\{ Y_{n+1} \geq k \} & \text{ if }k = \#\{H_i^{(n+1)}\}
\end{cases}
\]
By definition of $A_j^{(n+1)}$ and \eqref{eq:examu1}, it follows
that we have mapped ${\mathcal F}_{n+1}$ into a partition in
$\mathcal A_{n+1}$. Moreover, $\mathbb P(A_j^{(n+1)})>0$ as a
consequence of \eqref{eq:mathbbPonFnYn}. The restriction of
$\mathbb P$ to the embedding of $ {\mathcal F}_{n+1}$ defines a
probability measure on ${\mathcal F}_{n+1}$ with $\mu_{n+1}(F)>0$ if
$F\neq\varnothing$. Note that $\mu_{n+1}$ is by construction an
extension of $\mu_n$ to ${\mathcal F}_{n+1}$.

Finally, the extension Theorem ensures the existence of the required
$\mu$, which is just mapped to the trace of $\mathbb P$ on the
embedded ${{\mathcal F}}_\infty $.
\end{example}

\section{Numerical Discrete Example}

\begin{example}[Coupon Collector]
Let $n$ objects $ \{e_1,\ldots,e_n\}$ be picked repeatedly with
probability $p_i$ that object $e_i$ is picked on a given try, with
$ \sum_{i}p_i=1$. Find the earliest time at which all $n$ objects
have been picked at least once.
\end{example}
It is not difficult to show that the general Coupon Collector's
Problem may be embedded into a Markow network of $N=2^n-1$--nodes
(see, \cite{A06}).

Thus, let $P$ be a $N\times N$-transition matrix on the state set
$X=\{1,\ldots,N\}$.

For any $X/\pi=\{\bold{x_1},\ldots,\bold{x_n}\}$ ($n\leq N$), we
define the $N\times n$-matrix
$$
Q_{ij} =
\left\{%
\begin{array}{ll}
    1, & \text{if $i\in \bold{x_j}$;} \\
    0, & \text{otherwise.} \\
\end{array}%
\right.
$$
Then $P_\cdot = PQ$ is a nonnegative $N\times n$-matrix. It is a
transition probability matrix from $X$ to $X/\pi$ (we called it
$P(x,A_j^{(\cdot)})$). Each row $i$ represents the restriction of
$P(i,\cdot)$ to $\{\pi^{-1}(\bold{x_j}),j=1,\ldots,n\}$. As noted in
the proof of  Theorem~\ref{thm:mainta}, we should choose a
probability measure $\mu$ on $X$ and then define a new matrix $\hat P$ on
$\sigma({X/\pi})$ with \eqref{eq:def_of_P_n}. A
``neutral'' choice for $\mu$ is $\mu(i)=1/N$. Accordingly, by
\eqref{eq:def_of_P_n}, for any $i,j\in \{1,\ldots,n\}$,
\[
\hat{P}(\bold{x_i},\bold{x_j}) = \sum_{l\colon \pi(l)\in \bold{x_i}}
\frac{P_\cdot(l,\bold{x_j})}{\frac{\#\{l\colon \pi(l)\in
\bold{x_i}\}}{N}} \frac{1}{N} = \sum_{l\colon \pi(l)\in \bold{x_i}}
\frac{P_\cdot(l,\bold{x_j})}{\#\{l\colon \pi(l)\in \bold{x_i}\}}.
\]
A simple computation gives
\[
\hat{P} = (Q^TQ)^{-1} Q^T P Q
\]
where $(Q^TQ)^{-1} = \text{diag}(\frac{1}{\#\{l\colon \pi(l)\in
\bold{x_1}\}}, \frac{1}{\#\{l\colon \pi(l)\in \bold{x_2}\}} ,\ldots,
\frac{1}{\#\{l\colon \pi(l)\in \bold{x_n}\}} )$.

We have tested two target algorithms on a coupon collector problem with
$n=18$ objects. In this case $P$ is given by a $2^{18}\times 2^{18}$-sparse matrix.
The computation of $P^n$ is not practicable. The number of components of each class of
$X/\pi_m$ is plotted for $m=0,1,2,3$ in the following Figure~\ref{Fig:coupon1}
and Figure~\ref{Fig:coupon2}. The sequence $(X/\pi_m)_m$ varies according to the
target algorithm. We remark that both strategies converge to the same exact solution.

\begin{figure}[ph]
  {\psfrag{a0}[cm][ct]{{\tiny{$X/\pi_0=\{T,X\setminus T\}$}}}
  \psfrag{a1}[cm][ct]{{\tiny{$X/\pi_1=\{T=A^{(1)}_1,A^{(1)}_2,A^{(1)}_3\}$}}}
  \psfrag{a2}[cm][ct]{{\tiny{$X/\pi_2=\{T=A^{(2)}_1,\ldots,A^{(2)}_{13}\}$}}}
  \psfrag{a3}[cm][ct]{{\tiny{$X/\pi_3=\{T=A^{(3)}_1,\ldots,A^{(3)}_{29}\}$}}}
  \includegraphics[width=0.95\textwidth]{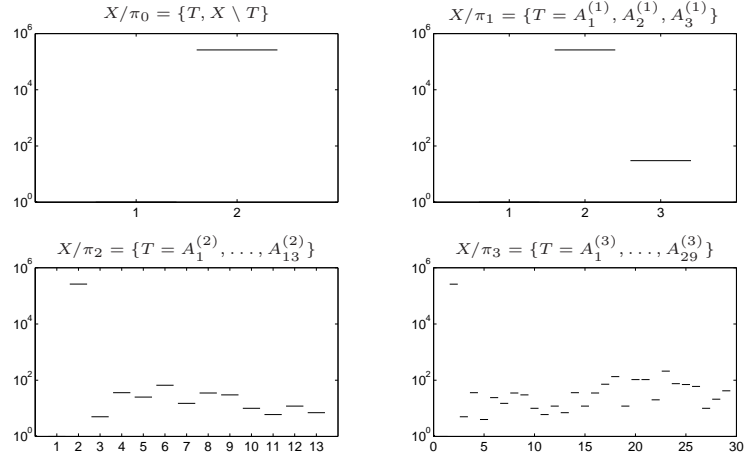}}\\
  \caption{Number of states for each class of $X/\pi_m$ (log scale), with
  $\epsilon_1 = 0.5$, $\epsilon_2 = 0.1$, $\epsilon_3 = 0.05$.
  The choice of $(\sim_{\epsilon_n})_n$
  is made as in Example~\ref{exa:e-cut}.}\label{Fig:coupon1}
\end{figure}
\begin{figure}[ph]
  {\psfrag{b0}[cm][ct]{{\tiny{$X/\pi_0=\{T,X\setminus T\}$}}}
  \psfrag{b1}[cm][ct]{{\tiny{$X/\pi_1=\{T=A^{(1)}_1,A^{(1)}_2,A^{(1)}_3\}$}}}
  \psfrag{b2}[cm][ct]{{\tiny{$X/\pi_2=\{T=A^{(2)}_1,\ldots,A^{(2)}_{12}\}$}}}
  \psfrag{b3}[cm][ct]{{\tiny{$X/\pi_3=\{T=A^{(3)}_1,\ldots,A^{(3)}_{30}\}$}}}
  \includegraphics[width=0.95\textwidth]{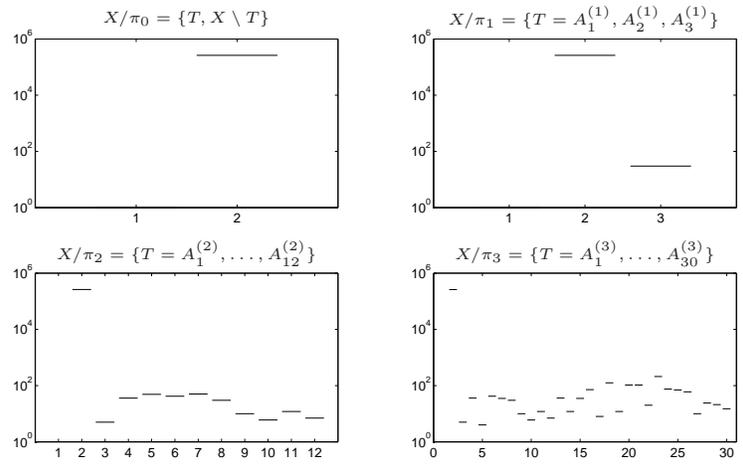}}\\
  \caption{Number of states for each class of $X/\pi_m$ (log scale)
  with $\epsilon_1 = 0.5$, $\epsilon_2 = 0.1$, $\epsilon_3 = 0.05$,
  with a different choice of the target algorithm.}\label{Fig:coupon2}
\end{figure}

\section{Blackwell}\label{sec:A0}

The problem of approximation is mathematically different if we start
from a Markov process with a countable set of states or with an
uncountable one. Let us consider, for the moment, the countable
case: $X$ is the at most countable set of the states and $\mathcal
X=2^X$ is the power set. Each function on $X$ is measurable. If we
take any equivalence relation on $X$, it is both measurable and
identified by the $\sigma$-algebra it induces (see
Theorem~\ref{thm:why_A0} below). This is not in general the case
when we deal with a measurable space $(X,\mathcal X)$, with $X$
uncountable.  In this section, we want to connect the process of
approximation with the upgrading information. More precisely, a
measurable equivalence $\pi=\pi_f$ defines both the partition
$X/\pi$ and the sigma algebra $\sigma(f)$. One wishes these two
objects to be related, in the sense that ordering should be
preserved. Example~\ref{exa:Why_A0}  below shows a paradox
concerning $\pi_f$ and $\sigma(f)$ when $X$ is uncountable.
In fact,
\begin{lemma}\label{lem:ordering_1}
Let $\mathcal A_1\subseteq\mathcal A_2$ be countably generated sub
$\sigma$-algebras of a measurable space $(X,\mathcal X)$. Then
$[x]_{\mathcal A_1}\supseteq [x]_{\mathcal A_2}$.

In particular, let $f,g$ be random variables. If
$\sigma(f)\supseteq\sigma(g)$, then $\pi_{f}\subseteq \pi_{g}$.
\end{lemma}
\begin{proof} See Appendix \ref{app:measuring}.
\end{proof}
 The problem is that even if a partition is more
informative than another one, it is not true that it generates a
finer $\sigma$-algebra, \emph{i.e.}, the following implication is
not always true for any couple of random variables $f$ and $g$
\begin{equation}\tag{\textrm{A0}}\label{eq:A0def}
 \pi_f\subseteq\pi_g \Longrightarrow
\sigma(f)\supseteq\sigma(g).
\end{equation}
Then Lemma~\ref{lem:ordering_1} is not invertible, if we do not
require the further Assumption~\eqref{eq:A0def} on the measurable
space $(X,\mathcal X)$. This last fact connects the space
$(X,\mathcal X)$ with the theory of Blackwell spaces (see
Lemma~\ref{lem:A0_A1}). We will assume the sole
Assumption~\eqref{eq:A0def}.
\begin{example}[$\pi_f=\pi_g \nRightarrow \sigma(f)=\sigma(g)$]
We give here a counterexample to Assumption~\eqref{eq:A0def}, where
two random variables $f,g$ generate two different sigma
algebras $\sigma(f)\neq\sigma(g)$ with the same set of atoms.
Obviously, Assumption~\eqref{eq:A0def} does not hold.

Let $(X,{\mathcal B}_X)$ be a Polish space and suppose ${\mathcal
B}_X \subsetneq \mathcal X$. Let $A \in \mathcal X \setminus
{\mathcal B}_X$ and consider the sequence $\{A_n,n\in\NN\}$ that
determines ${\mathcal B}_X$, i.e.\ ${\mathcal
B}_X=\sigma(A_n,n\in\NN)$. Let $\mathcal A=\sigma(A, A_n, n\in\NN)$.
 ${\mathcal B}_X\subsetneq \mathcal A $. As a consequence
of Lemma~\ref{lem:count_gen_new}, there exist two random variables
$f,g$ such that ${\mathcal B}_X=\sigma(f)$ and
$\mathcal A=\sigma(g)$. The atoms of ${\mathcal B}_X$ are the points
of $X$, and then the atoms of ${\mathcal A}$ are also the points of
$X$, since ${\mathcal B}_X\subseteq \mathcal A$.
\end{example}

We recall here the definition of Blackwell spaces. A measurable
space $(X,\mathcal X)$ is said \emph{Blackwell} if $\mathcal X$ is a
countably generated $\sigma$-algebra of $X$ and $\mathcal A=\mathcal
X$ whenever $\mathcal A$ is another countably generated
$\sigma$-algebra of $X$ such that $\mathcal A\subseteq \mathcal X$,
and $\mathcal A$ has the same atoms as $\mathcal X$. A metric space
$X$ is Blackwell if, when endowed with its Borel $\sigma$-algebra,
it is Blackwell. The measurable space $(X,\mathcal X)$ is said to be
a \emph{strongly Blackwell space} if $\mathcal X$ is a countably
generated $\sigma$-algebra of $X$ and
\begin{itemize}
 \item[(\textrm{A1})]
$\mathcal A_1=\mathcal A_2$ if and only
if the sets of their atoms coincide, where
$\mathcal A_1$ and $\mathcal A_2$ are
countably generated $\sigma$-algebras with
$\mathcal A_i\subseteq\mathcal X$ $i=1,2$.
\end{itemize}

For what concerns Blackwell spaces, the literature is quite
extensive. D.~Bla\-ck\-well proved that every analytic subset of a
Polish space is, with respect to its relative Borel $\sigma$-field,
a strongly Blackwell space (see \cite{black1956}). Therefore, if
$(X,\mathcal B_X)$ is (an analytic subset of) a Polish space and
 $\mathcal B_X\subsetneq \mathcal X$, then $(X,\mathcal X)$ cannot
be a weakly Blackwell space. To see this, take $A_1,A_2\ldots$ a
base of $\mathcal B_X$ and $A\in\mathcal X\setminus\mathcal B_X$.
Then $\mathcal B_X=\sigma(A_1,A_2\ldots)$ and $\mathcal
A=\sigma(A,A_1,A_2\ldots)$ have the same set of atoms (the points of
$X$) but ${\mathcal A}\nsubseteq {\mathcal B}_X$ (or, equivalently,
the identity function $I_d:(X,\mathcal B_X)\to(X,\mathcal A)$ is not
measurable).
Moreover, as any (at most) countable set equipped with any
$\sigma$-algebra may be seen as an analytic subset of a
Polish space, then it is a strongly Blackwell space.

A.~Maitra exhibited coanalytic sets that are not Blackwell spaces
(see \cite{Maitra70}). M.~Orkin constructed
 a nonanalytic (in fact nonmeasurable) set in a Polish space that is a Blackwell space (see \cite{Orkin72}).
Jasi\'nski showed (see \cite{Jasinski85}) that continuum hypothesis
(CH) implies that there exist uncountable Sier\-pi\'n\-ski and Luzin
subsets of $\R$ which are Blackwell spaces (implying in a strong
way that Blackwell spaces do not have to be Lebesgue measurable or
have the Baire property). Jasi\'nski also showed that CH implies
that there exist uncountable
 Sierpi\'nski and Luzin subsets of $\R$ which are not  Blackwell spaces
 (implying in a strong way that Lebesgue measurable sets and sets with the
Baire property
  do not have to be Blackwell spaces). This latter result is strengthened by
R.M.~Shortt in
  \cite{Shortt87} by showing that CH implies the existence of uncountable
Sierpi\'nski and Luzin
  subsets of $\R$ which are highly non-Blackwell in the sense that all
Blackwell subspaces of the
  two sets are countable.

Note that Assumption~\eqref{eq:A0def} and
Assumption~(A1) coincide, as the following Lemma states.
\begin{lemma}\label{lem:A0_A1}
Let $(X,\mathcal X)$ be a measurable space.
Then \eqref{eq:A0def} holds if and only
if {\upshape{(A1)}} holds.
\end{lemma}
\begin{proof}
Lemma \ref{lem:count_gen_new} in appendix  states
that $\mathcal A\subseteq\mathcal X$ is countably generated if and
only if there exists a random variable $f$ such that $\mathcal
A=\sigma(f)$. In addition, as a consequence of
Lemma~\ref{lem:ordering_1}, we have only to prove that (A1) implies
\eqref{eq:A0def}. By contradiction, assume (A1), $\pi_f\subseteq
\pi_g$, but $\sigma(g)\nsubseteq\sigma(f)$. We have
$\sigma(f,g)\neq\sigma(f)$, and then $\pi_{\sigma(f,g)}\neq\pi_f$ by
(A1) and Lemma~\ref{lem:count_gen_new}. On the other hand, as a
consequence of Eq.~\eqref{eq:vee_and_cap}, we have that
$\pi_{\sigma(f,g)}=\pi_{\sigma(f)\vee\sigma(g)}=\pi_f\cap\pi_g=
\pi_f$.
\end{proof}
We call \emph{weakly Blackwell space} a measurable space
$(X,\mathcal X)$ such that Assumption~\eqref{eq:A0def} holds. If
$(X,\mathcal X)$ is a weakly Blackwell space, then $(X,\mathcal F)$
is a weakly Blackwell space, for any $\mathcal F\subseteq \mathcal
X$. Moreover, every strong Blackwell space is both a Blackwell space
and a weakly Blackwell space whilst the other inclusions are not
generally true. In \cite{Jasiski-1985,Shortt-1987}, examples are
provided of Blackwell spaces which may be shown not to be weakly Blackwell. The
following example shows that a weakly Blackwell space need not be
Blackwell.
\begin{example}[weakly Blackwell $\nRightarrow$ Blackwell]
Let $X$ be an uncountable set and $\mathcal X$ be the
countable--cocountable $\sigma$-algebra on $X$. $\mathcal X$ is
easily shown to be not countably generated, and therefore
$(X,\mathcal X)$ is not a Blackwell space. Take any countably
generated $\sigma$-field $\mathcal A\subseteq \mathcal X$, i.e.\
$\mathcal A=\sigma(\{A_i,i\in\mathbb N\})$.
\begin{itemize}
 \item Since each set (or its complementary) of $\mathcal X$ is countable,
then, without loss of generality, we can assume the cardinality
of $X\setminus A_i$ to be countable.
 \item Each atom $B$ of $\sigma(A_i,i\in\mathbb N)$ is of the form
\begin{equation}\label{eq:count-cocount}
 B = \cap_{i=1,2,\ldots} C_i,
\qquad \text{where $C_i=A_i$ or $C_i=X\setminus A_i$, for any $i$.}
\end{equation}
\end{itemize}
Note that the cardinality of the set
$A:=\cup_{i} (X\setminus A_i)$ is countable,
as it is a countable union of countable sets. As a consequence of
\eqref{eq:count-cocount}, we face two types of atoms:
\begin{enumerate}
 \item for any $i$, $C_i=A_i$. This is the atom made by the intersections
of all the uncountable generators. This is an uncountable atom, as
it is equal to $X\setminus A$.
 \item exists $i$ such that $C_i=X\setminus A_i$. This implies that this atom
is a subset of the countable set $A$. Therefore, all the atoms
(except $X\setminus A$) are disjoint subsets of the countable set
$A$ and hence they are countable.
\end{enumerate}
It follows that the number of atoms of $\mathcal
A$ is at most countable. Thus, $(X,\mathcal A)$ is a strongly Blackwell
space, i.e.\ $(X,\mathcal X)$ is a weakly Blackwell space.
\end{example}

\begin{example}[Information and $\sigma$-algebra
(see \cite{Dubra-Echenique-2004})]\label{exa:Why_A0} Suppose
$X=[0,1]$, $\mathcal X=\sigma(\mathcal Y,A)$ where $\mathcal Y$ is
the countable--cocountable $\sigma$-algebra on $X$ and $A=[0,1/2)$.
Consider a decisionmaker who chooses action 1 if $x < \frac{1}{2}$
and action 2 if $x \geq \frac{1}{2}$. Suppose now that the
information is modeled either as the partition of all elements of
$X$, $\tau=\{x, x \in X\}$ and in this case the decisionmaker is
perfectly informed, or as the partition $\tau'=\{A, X\setminus A\}.$
If we deal with $\sigma$-algebras as a model of information then
$\sigma(\tau)=\mathcal Y$ and $\sigma(\tau')=\sigma(A)$. The
partition $\tau$ is more informative than $\tau'$, whereas
$\sigma(\tau)$ is not finer than $\sigma(\tau')$. In fact
$A\not\in\mathcal Y$ and therefore if the decisionmaker uses
$\sigma(\tau)$ as its structure of information, believing it more
detailed than $\sigma(\tau')$, he will never know whether or not the
event $A$ has occurred and can be led to take the wrong decision. In
this case, $\sigma$-algebras do not preserve information because
they are not closed under arbitrary unions. However, if we deal with
Blackwell spaces, any countable $\sigma$-algebra is identified by
its atoms and therefore will possess an informational content (see
\cite{stinchcombe}, for example).
\end{example}

The following theorem, whose proof is in Appendix
\ref{app:A0}, links the measurability of any relation with
the cardinality of the space and Assumption~\eqref{eq:A0def}. It
shows the main difference between the uncountable case and the
countable one.
\begin{theorem}\label{thm:why_A0}
Assume {\upshape{(CH)}}. Let $(X,\mathcal X)$ be a
measurable space. The following
properties are equivalent:
\begin{enumerate}
 \item\label{co:3} Any equivalence relation $\pi$
on $X$ is measurable and Assumption~\eqref{eq:A0def} holds;
 \item\label{co:2} $(X,2^X)$ is a weakly Blackwell space;
 \item\label{co:1} $X$ is countable and $\mathcal X=2^X$.
\end{enumerate}
\end{theorem}

\section{Proofs}\label{sec:proofs}
The following theorem mathematically motivates our approximation
problem: any limit of a monotone sequence of discrete equivalence
relationships is a measurable equivalence.
\begin{theorem}\label{limit_eq}
For all $n\in\NN$, let $\pi_n$ be a discrete equivalency.
Then $\pi_\infty = \cap_n \pi_n$ is a measurable equivalency.
Conversely, for any measurable equivalency $\pi$, there
exists a sequence $(\pi_n)_n$ of discrete equivalencies such that
$\pi_\infty = \cap_n \pi_n$.
\end{theorem}
\begin{proof} See Appendix \ref{app:measuring}.
\end{proof}
\begin{proof}[Proof of Theorem \ref{thm:CP+TP}]
Let $\pi=\pi_f$ be a compatible projection. We define
\[
P_\pi(x,F):=P(x,F),\qquad \forall x\in X,
    \forall F\in \sigma(f).
\]
What remains to prove is that
$P_\pi\in\mathrm{TrP}(X,\mathcal X,\sigma(f))$. More precisely,
we have to show that $P_\pi (\cdot,F)$ is $\sigma(f)$-measurable,
$\forall F\in \sigma(f)$. By contradiction, there exists $F\in \sigma(f)$
such that the random variable $Y^F(\omega)=P_\pi(\omega,F)$ is not
$\sigma(f)$-measurable. Then $\sigma(Y^F)\nsubseteq\sigma(f)$, and hence
$\pi_{Y^F}\nsupseteq \pi_f$ by Assumption~\eqref{eq:A0def},
which contradicts Equation~\eqref{defn:comp_proj}.
\end{proof}

\begin{proof}[Proof of Theorem \ref{thm:pi_infty}]
As a consequence of Theorem~\ref{limit_eq}, $\pi_\infty = \pi_f$,
where $\sigma(f)=\vee_n \mathcal F_n$.
Define
\[
P_{\infty}(x,F):=P(x,F),\qquad \forall x\in X,
    \forall F\in \sigma(f).
\]
We will prove that, for any $F\in\sigma(f)$,
$P_{\infty}(\cdot,F)$ is $\sigma(f)$-measurable and consequently $\pi_{\infty}$ will be a compatible projection. This implies that there
exists a measurable function $h_F:(\R,\mathcal B_R)\to(\R,\mathcal B_R)$ so
that $P_{\infty}(\omega,F) = h_F(f(\omega))$. Therefore, if $x\,\pi_f\,y$,
then $P_{\infty}(x,F)=P_{\infty}(y,F)$, which is the thesis.

Thus, we show that for
any $F \in \sigma(f)$ and $t\in\mathbb R$, we have
\begin{equation}\label{eq:proof_F-i-m1}
H:=\{x\colon P(x,F)\leq t\} \in \sigma(f)
\end{equation}
To prove Equation~\eqref{eq:proof_F-i-m1},
we first show that it is true when $F\in \mathcal F_n$ by proving
that
$$
H =\cap_{m > n}\pi^{-1}_m\pi_m(H),
$$
which implies that $H \in \sigma(f)$. The inclusion $H
\subseteq\cap_m\pi^{-1}_m\pi_m(H)$ is always true. For the other
inclusion, let $y \in \cap_{m > n}\pi^{-1}_m\pi_m(H)$. Let $m
> n$; there exists $x_m\in H$ such that $y \pi_m x_m$.
Therefore, Equation~\eqref{eq:projectionsta} and the definition of
$\sim_{\epsilon_n}$ imply $P(y,F)\leq P(x_m, F) + \epsilon_m\leq t
+\epsilon_m$, for any $m > n$. As $\epsilon_m \searrow 0$, we obtain
that $y \in H$. Then Equation~\eqref{eq:proof_F-i-m1}
is true on the algebra $Alg:=\cup_n\mathcal F_n$.\\

Actually, let $F_n \in Alg$ such that $F_n \nearrow F$. We prove
that Equation~\eqref{eq:proof_F-i-m1} holds for $F$ by showing
that
$$
H =\{x\colon P(x,F)\leq t\} =\cap_n \{x\colon P(x,F_n)\leq t\} =:
\cap_n H_n .
$$
Again, since $F_n\subseteq F$, then $P(x,F_n)\leq P(x,F)$ and
therefore $H\subseteq \cap_n H_n $. Conversely, the set $\cap_n
H_n\setminus H$ is empty since the sequence of $\mathcal
X$-measurable maps $P(\cdot,F)-P(\cdot,F_n)$ converges to $0$:
\[
P(\cdot,F)-P(\cdot,F_n) = P(\cdot,F\setminus F_n) \to
P(\cdot,\varnothing) = 0.
\]
Then Equation~\eqref{eq:proof_F-i-m1} is true on the monotone
class generated by the algebra $Alg=\cup_n\mathcal F_n$, i.e.,
Equation~\eqref{eq:proof_F-i-m1} is true on $\sigma(f)$.
\end{proof}
\begin{proof}[Proof of Theorem \ref{thm:maintb}]
Given a target algorithm $(X,\mathcal F_n,T,P_n)_n$,
let $\pi_\infty=\pi_f $ be defined as in
Theorem~\ref{thm:pi_infty}.
We show that $\pi_\infty$ is optimal.
Let $\psi_g$ be another compatible projection and let $(X,\sigma(g),T,P_g)$
be the target problem given by Theorem~\ref{thm:CP+TP}. We are
going to prove by induction on $n$ that
\begin{equation}\label{eq:optim_n}
\forall n\in\NN, \qquad \mathcal F_n\subseteq \sigma(g) .
\end{equation}
In fact, for $n=0$ it is sufficient to note that
$\mathcal F_0=\sigma(\{T\})\subseteq \sigma(g)$.

Equation~(\ref{eq:projectionsta}) states that $\mathcal F_n =
\sigma(\mathcal F_{n-1},h_n)$, where $h_n$ is the discrete random variable,
given by Lemma~\ref{lem:discrete_measur}, s.t.\
\begin{align*}
x\,&\pi_{h_n}\,y  \\
& \Updownarrow \\
\Big(\big(P(x,A^{(n-1)}_i)_i\big)
& \sim_{\epsilon_n}\big(P(y,A^{(n-1)}_i)_i\big)\Big)
\end{align*}

Let $k^{(n-1)}_i:X\to [0,1]$ be defined as $ k^{(n-1)}_i(x) =
P(x,A^{(n-1)}_i)$.  Then
\[
\begin{diagram}
  \node{X} \arrow{seee,l}{
  h_n}\arrow{s,l}{
  (k^{(n-1)}_i)_{i\in\NN}}
\\
  \node{S} 
\arrow[3]{e,r}{\epsilon_n} 
  \node[3]{S/\epsilon_n}
\end{diagram}
\]
Obviously, $\sigma(h_n)\subseteq
\sigma(k^{(n-1)}_1,k^{(n-1)}_2,\ldots)$. For the induction step, as
$A^{(n-1)}_i \in\mathcal F_{n-1}\subseteq \sigma(g)$, we have that $
P_g(\cdot,A^{(n-1)}_i) $ is $\sigma(g)$-measurable, and therefore
$\sigma(k^{(n-1)}_i)\subseteq\sigma(g)$. Then
$\mathcal F_{n}=\sigma(\mathcal F_{n-1},h_n)\subseteq
\sigma(\mathcal F_{n-1},k^{(n-1)}_1,k^{(n-1)}_2,\ldots)
\subseteq\sigma(g)$. Therefore
$\sigma(f)=\vee_n\mathcal F_{n}\subseteq\sigma(g)$, which implies
$\pi_\infty\supseteq \psi_g$ by Lemma~\ref{lem:ordering_1},
and hence $\pi_\infty$ is optimal.
\end{proof}

\begin{corollary}\label{cor:uniqueness}
$\pi_\infty$ does not depend on the choice of $\mathrm{Str}$.
\end{corollary}
\begin{proof}
$\pi_\infty=\cap_n\pi_n$ is optimal, $\forall (\pi_n)_n =
\mathrm{Str}(P)$. The optimal projection being unique,
we are done.
\end{proof}

\begin{proof}[Proof of Theorem \ref{thm:mainta}]
Let $\pi_\infty=\pi_f $ be defined as in
Theorem~\ref{thm:pi_infty} and $(X,\sigma(f),T,P_\infty)$ be
given by Theorem~\ref{thm:CP+TP} so that
\(P(x,F)=P_\infty(x,F)\) for any $F\in\sigma(f)$.
Then each $(P_n)_n$
of Definition~\ref{eq:def_of_P_n} can be rewritten as
\begin{equation}\label{eq:def_of_P_n2}
{P}_n(x,F)= \frac{\int_{[x]_n}
P_\infty(x,F)\mu(dz)}{\mu([x]_n)}
, \qquad
\forall x\in X, \forall F\in {\mathcal F}_n,
\end{equation}
where $[x]_n$ is the $\pi_n$-class of equivalence of
$x$ and $\mu([x]_n)>0$ since $[x]_n\neq\varnothing$.

Note that $d(P,{P}_m)\leq 2\sum_n\beta^n$.
Then, for any $\epsilon>0$, there
exists an $N$ so that $\sum_{n>N} \beta^n\leq \frac{\epsilon}{2}$.
Therefore we are going to prove by induction on $n$ that
$$
\sup_x \vert {P}^n_m(x,T)-P^n(x,T)\vert \rightarrow 0 \text{
as } m \text{ tends to infinity},
$$
which completes the proof.
If $n=1$, then by definition of $\epsilon_m$, since
$T\in {\mathcal F}_{m-1}$, we have that
 \begin{align*}
|{P}_m(x,T)-P(x,T)
|
&
\leq \frac{\int_{[x]_n}|P_\infty(z,T)
-P(x,T)|\mu(dz)}{\mu([x]_n)}
\\
&
= \frac{\int_{[x]_n}|P(z,T)
-P(x,T)|\mu(dz)}{\mu([x]_n)}
\\
& \leq \epsilon_m \frac{\int_{[x]_n}{\mu(dz)}}{\mu([x]_n)}
= \epsilon_m .
 \end{align*}
For the induction step, we note that
 \begin{multline}\label{eq:prth1IH0}
\big\vert {P}^{n+1}_m(x,T)-P^{n+1}(x,T)\big\vert \\
\leq
\sum_i \Big\vert {P}_m (x ,A_i^{(m)})
{P}^n_m(A_i^{(m)},T) - \int_{A_i^{(m)}} P(x, dz )
P^n (z,T) \Big\vert ,
 \end{multline}
where $(A_i^{(m)})_i$ is the partition of $X$ given by $\pi_m$. By
induction hypothesis,
\begin{equation*}
| {P}^n_m(z,T) - P^n (z,T)| \leq \epsilon
\end{equation*}
for $m\geq m_0$ large enough. Since $[z]_m = A_i^{(m)}$ if $z\in
A_i^{(m)}$, it follows that
 \[
\int_{A_i^{(m)}} P(x , dz ) \Big|{P}^n_m(A_i^{(m)},T) - P^n (z,T)
\Big| \leq \epsilon \int_{A_i^{(m)}}
P(x , dz ) .
 \]
Equation~\eqref{eq:prth1IH0} becomes
 \begin{multline*}
\big\vert {P}^{n+1}_m(x,T)-P^{n+1}(x,T)\big\vert \\
\leq \epsilon + \sum_i {P}^n_m(A_i^{(m)},T) \big\vert {P}_m (x,A_i^{(m)})
- P(x , A_i^{(m)} ) \big\vert \\
\leq \epsilon + \sum_i \big\vert {P}_m (x,A_i^{(m)})
- P(x , A_i^{(m)} ) \big\vert.
  \end{multline*}
On the other hand, by Equation~\eqref{eq:def_of_P_n2},
\[
{P}_m (x,A_i^{(m)}) - P(x , A_i^{(m)} ) =
\int_{[x]_m}\frac{P_\infty(z,A^{(m)}_i)-P(x,A^{(m)}_i)}{\mu([x]_m)}\mu(dz).
\]
The definition of $\sim_{\epsilon_{m+1}}$ states that
 \[
\sum_i \big\vert P_\infty(z,A^{(m)}_i)-P(x,A^{(m)}_i) \big\vert \leq
\epsilon_{m+1}
 \]
whenever $z\in [x]_m$ and therefore
 \begin{multline*}
\big\vert {P}^{n+1}_m(x,T)-P^{n+1}(x,T)\big\vert \\
\leq \epsilon + \int_{[x]_m}\sum_i \big\vert
P_\infty(z,A^{(m)}_i)-P(x,A^{(m)}_i)\big\vert
\frac{\mu(dz)}{\mu([x]_m)} \leq \epsilon + \epsilon_{m+1}.
  \end{multline*}
Since $\epsilon_m \rightarrow 0$ as $m$ tends to
infinity, we get the result.
\end{proof}

\begin{proof}[Proof of Theorem \ref{thm:maintc}]
By Definition~\ref{eq:def_of_P_n} and Lemma~\ref{lem:contmeas},
$(P_n(\cdot,F))_{n\geq m}$ is
a martingale with respect to the filtration $(\mathcal F_n)_{n\geq m}$,
for any $F\in\mathcal F_m$. Then, if $Y^F(x)=P(x,F)$ as in
Definition~\ref{eq:def_of_P_n}, we have that
\[
P_n(x,F) \mathop{\longrightarrow}_{n\to\infty}
\EE_{\mu}[Y^F\vert \vee_n{\mathcal F}_{n}](x) = Y^F(x), \qquad
\text{for $\mu$-a.e. } x\in X, \forall F\in \cup_m {\mathcal F}_m.
\]
Let $\pi_\infty=\pi_f $ be defined as in
Theorem~\ref{thm:pi_infty} and $(X,\sigma(f),T,P_\infty)$
given by Theorem~\ref{thm:CP+TP} so that
\(P(x,F)=P_\infty(x,F)\) for any $F\in\sigma(f)$.
Then
\[
P_n(x,F) \mathop{\longrightarrow}_{n\to\infty} P_\infty(x,F)
\]
for any $x\in X$ and $F\in \cup_m {\mathcal F}_m$, the only
$\mu$-null set in $\cup_m {\mathcal F}_m$ being the empty set.
\end{proof}

\subsection{Weak convergence of conditional probabilities}

Let the target problem $(X,\mathcal F,T,P)$ be given and let
$\mathrm{Str}=(\mathrm{Str}_n)_n$,
where $\mathrm{Str}_n(X,\mathcal F,T,P)=
(X,\tilde{\mathcal F}_n,T,\tilde P_n)$ be a target algorithm.
In order to prove Theorem \ref{prop:Top}, which states the total
convergence of the probability measure ${P}_n(x,\cdot)$ towards
$P(x,\cdot)$, we proceed as follows:
\begin{itemize}
    \item first, we define the topology $\varrho_P$ on $X$;
    \item then, we define a ``natural'' topology
    $\tau_{\mathrm{Str}}$ on $X$ associated to any target algorithm
    $({\mathrm{Str}_n})_n$.
    We prove in Theorem~\ref{thm:tot_conv} the total convergence
    of $(P_n)_n$ to $P_\infty$, under this topology;
    \item then, we define the topology $\tau_P$ on $X$ as the intersection
    of all the topologies $\tau_{\mathrm{Str}}$;
    \item finally, we show Theorem~\ref{prop:Top}
    by proving that $\varrho_P\subseteq\tau_{\mathrm{Str}}$.
    The non-triviality of $\varrho_P$ will imply that of $\tau_P$.
\end{itemize}

We introduce the pseudometric $d_P$ on $X$ as follows:
\[
d_P ( x,y) = \sum_n\beta^n \big\vert P^n(x,T) - P^n(y,T)\big\vert.
\]

Now, let $\tau_{\mathrm{Str}}$ be
the topology generated by $\cup_n {\mathcal F}_n$.
$C$ is a closed set if and only if $C = \cap_n C_n, C_n \in
{{\mathcal F}_n}$. In fact,
if $C \in {\mathcal F}_n$, for a given $n$, then $C
\in {\mathcal F}_{n+p}$, for any $p$ and therefore $C$ is closed.
$(X,\tau_{\mathrm{Str}})$ is a topological space.
\begin{remark}
Let us go back to Example~\ref{exa:dyadic}. The topology defined by
asking that the sets in each $\mathcal F_n$ are closed is strictly
finer than the standard topology. On the other hand, the same
example may be explained with left closed--right opened dyadic
subdivisions, which leads to a different topology that also contains
the natural one. Any other ``reasonable'' choice of subdivision will
show the same: the topologies are different, and all contain the
standard one. In the same manner, we are going to show that all the
topologies $\tau_{\mathrm{Str}}$ contain the standard one,
$\varrho_P$.
\end{remark}

\begin{theorem}\label{thm:tot_conv}
Let the target problem $(X,\mathcal F,T,P)$ and the target algorithm
$(X,\mathcal F_n,T,P_n)_n$ be given.
For any target algorithm $\mathrm{Str}$,
\[
{P}_n(x,
\cdot)\mathop{\longrightarrow}\limits^{\mathrm{tot}}_{\tau_{\mathrm{Str}}}
P(x, \cdot), \qquad \forall x\in X.
 \]
\end{theorem}
\begin{proof}
Let $C =\cap_nC_n$ be a closed set of $\mathrm{Str}$ and
let $\bar{P}_n$ be any extension
of ${P}_n$ to $\vee_n\mathcal F_n$.
We have to check that ${\limsup}_n \bar{P}_n(x,C)
\leq P(x,C)$, for any given $x$ (see, e.g.,
\cite{Billingsley:book}). Note that, since $C\in\vee_n\mathcal F_n$,
we have $P(x,C)=P_\infty( x,C)$. But, $\bar{P}_n(x,C)-P_\infty(
x,C) \leq \bar{P}_n(x,C_{n-1})-P_\infty( x,C)=
{P}_n(x,C_{n-1})-P_\infty( x,C)$. Actually,
\begin{multline*}
{P}_n(x,C_{n-1})-P_\infty( x,C)\\
= \Big(\underbrace{{P}_n(x,C_{n-1})-P_\infty( x,C_{n-1})}_{I}\Big)+
\Big(\underbrace{P_\infty( x,C_{n-1})- P_\infty( x,C)}_{II}\Big).
\end{multline*}
$I\rightarrow 0$ as $n$ tends to infinity, from the target
algorithm, since as $\epsilon_n \searrow 0$. For any $\epsilon >0$,
there exists $N_1>0$, such that for any $n \geq N_1$, $\vert
{P}_n(x,C_{n-1})-P_\infty( x,C_{n-1})\vert \leq
\frac{\epsilon}{2}$. \hfill\\\null $II\rightarrow 0$ as $n$ tends to
infinity, from the continuity of the measure. For any $\epsilon >0$,
there exists $N_2>0$ and for any $n \geq N_2$,
$\vert P_\infty( x,C_{n-1})- P_\infty( x,C)\vert \leq
\frac{\epsilon}{2}.$
\end{proof}

An example of a natural extension of $P_n$ to $\bar{P}_n$ is given
by
\[
\bar{P}_n(x,F)= \EE_{\mu}[Y^F\vert \vee_n\mathcal F_n](x), \qquad
\forall x\in X, \forall F\in \vee_n\mathcal F_n,
\]
where, for any $F\in\vee_n\mathcal F_n$, $Y^F$ is the $\vee_n{\mathcal F}_{n}$-random variable such that $Y^F(\omega)=
P_{\infty}(\omega,F)$. As mentioned for ${P}_n$,
$\bar{P}_n(x,\cdot)$ is a probability measure, for any $x\in X$.

\begin{corollary}
For any fixed strategy $\mathrm{Str}(P)$, let $P_n$ be as in
Theorem~\ref{thm:mainta}. We have
\[
{P}_n(x, \cdot)
\mathop{\longrightarrow}\limits^{\mathrm{tot}}_{\tau_P} P(x,
\cdot),
 \]
for any given $x$.
\end{corollary}
In order to describe the topology $\tau_P$, we will denote by
$[[F]]_*$ the closure of a set $F\subseteq X$ in  a given topology
$*$. Note that the monotonicity of $\pi_n$ implies
\[
[[F]]_{\tau_{\mathrm{Str}}} = \cap_n
[[F]]_{\tau_{\mathrm{Str}_n}}
\]
where $\tau_{\mathrm{Str}_n}$ is the (discrete) topology on
$X$ generated by $\mathcal F_n$. Since $\tau_P$ is the intersection
of all the topologies $\tau_{\mathrm{Str}}$, we have
\[
[[F]]_{\tau_P} \supseteq [[F]]_{\tau_{\mathrm{Str}}} =
\cap_n [[F]]_{\tau_{\mathrm{Str}_n}} ,
\qquad \forall F\in 2^X,\forall
\mathrm{Str}.
\]
\begin{proof}[Proof of Theorem \ref{prop:Top}]
Let $F$ be the closed set in $\varrho_P$ so defined
\[
F := \{y\in X\colon d_P(y,x)\geq r\},
\]
i.e., $F$ is the complementary of an open ball in $(X,d_P)$ with
center $x$ and radius $r$. If we show that $F\in\tau_P$, then we
are done, since the arbitrary choice of $x$ and $r$ spans a base
for the topology $\varrho_P$.

We are going to prove
\[
F = [[F]]_{\tau_{\mathrm{Str}}} =
\cap_m [[F]]_{\tau_{\mathrm{Str}_m}} , \qquad \forall
\mathrm{Str},
\]
which implies \( [[F]]_{\tau_P} = F \). It is always true that
$F\subseteq[[F]]_*$; we prove the nontrivial inclusion $F \supseteq
\cap_m [[F]]_{\tau_{\mathrm{Str}_m}}$. Assume that $y\in
[[F]]_{\tau_{\mathrm{Str}}}$.
Now, $y\in
[[F]]_{\tau_{\mathrm{Str}_m}}$,
for any $m$, and then there exists
a sequence $(y_m)_m$ with $y_m\in
F$ such that $y \, \pi_m \, y_m$, for any $m$.
Thus, $y\in \cap_m [y_m]_m$, where $[x]_m$ is the $\pi_m$-class
of equivalence of $x$. Thus
\[
P^n_m(y_m,T) = P^n_m(y,T) , \qquad \forall m,n
\]
since $P_m(\cdot,T)$ is $\mathcal F_m$-measurable.
By Theorem~\ref{thm:mainta}, for any $n\in\NN$,
\[
|P^{n}(y,T)-P_m^{n}(y,T)| + |P_m^{n}(y_m,T)-P^{n}(y_m,T)|
\mathop{\longrightarrow}_{m\to\infty} 0 .
\]
Now, let $N$ be such that $\sum_{n=N}^{\infty}\beta^n \leq
\frac{\epsilon}{4}$ and take $n_0$ sufficiently large s.t.
\[
\sum_{n=0}^N
|P^{n}(y,T)-P_{n_0}^{n}(y,T)| +
|P_{n_0}^{n}(y_{n_0},T)-P^{n}(y_{n_0},T)|
\leq
\frac{\epsilon}{2}
\]
We have
\begin{multline*}
d_P(y_{n_0},y) = \sum_n\beta^n \big\vert
P^n(y,T)-P^n(y_{n_0},T)\big\vert\
\\
\leq \sum_{n=0}^N \big\vert P^n(y,T)-P^n(y_{n_0},T)\big\vert\ +
2\sum_{n=N}^{\infty}\beta^n \\
\leq \sum_{n=0}^N \Big(|P^{n}(y,T)-P_{n_0}^{n}(y,T)| +
|P^n_{n_0}(y,T)-P^n_{n_0}(y_{n_0},T) |\\
+|P_{n_0}^{n}(y_{n_0},T)-P^{n}(y_{n_0},T)| \Big) +
2\frac{\epsilon}{4}
\leq \frac{\epsilon}{2}
+\frac{\epsilon}{2} = \epsilon .
\end{multline*}
and therefore
\[
d_P(x,y) \geq d_P(x,y_{n_0})- d_P(y_{n_0},y) \geq r - \epsilon .
\]
The arbitrary choice of $\epsilon$ implies $y\in F$, which is the
thesis.
\end{proof}

\appendix
\numberwithin{theorem}{section}
\section{Results on equivalence relations}\label{app:measuring}

In this appendix we give the proof of auxiliary results that connect
equivalency with measurability.
\begin{proof}[Proof of Lemma~\ref{lem:discrete_measur}]
Let $\pi=\pi_f$ be a discrete equivalency on $X$. Then $X/\pi$ defines a
countable measurable partition of $X$. Conversely,
let $\{A_1, A_2,\ldots\}$ be a countable measurable partition
on $X$. Define $f:X\to\NN$ s.t.\ $f(x)=n \iff
x\in A_n$. Therefore $f$ is measurable and $\pi=\pi_f$ is a
discrete equivalency on $X$.
\end{proof}
\begin{lemma}\label{lem:mon_fun_alg_1}
Let $f,g$ be two random variables such that \( g(x)<g(y) \Rightarrow
f(x)<f(y) \). Then $\sigma(g)\subseteq \sigma(f)$.
\end{lemma}
\begin{proof}
Let $t\in\R$ be fixed. We must prove that $\{g \leq
t\}\in\sigma(f)$. If $\{g \leq t\}=$ or $\{g > t\}$ are empty, then
we are done. Assume then that $\{g \leq t\},\{g >
t\}\neq\varnothing$. We have two cases

- ${t^\ast\in f(\{g \leq t\})}$: $\exists x^\ast\in \{g \leq t\}$
such that $t^\ast = f(x^\ast)$.
\\
By definition of $t^\ast$, $\{g \leq t\}\subseteq\{f\leq t^*\}$.
Conversely, let $y\in\{g>t\}$. Since $g(x^\ast)\leq t<g(y)$, then
$f(x^\ast)= t^\ast<f(y)$, \emph{i.e.}\ $\{g > t\}\subseteq\{f>
t^*\}$. Then $\{g \leq t\}= \{f \leq t^\ast\}\in\sigma(f)$.

- ${t^\ast\not\in f(\{g \leq t\})}$: $\forall x\in \{g \leq t\}$ we
have that $f(x)<t^\ast$.
\\
Then $\{g \leq t\}\subseteq\{f< t^*\}$. Conversely, let
$y\in\{g>t\}$. Since $\forall x\in \{g \leq t\}$ $g(y)>g(x)$, then
$f(y)>f(x)$, which implies $f(y)\geq \sup f(\{g \leq t\})=t^\ast$,
\emph{i.e.}\ $\{g > t\}\subseteq\{f\geq t^*\}$. Then $\{g \leq t\}=
\{f < t^\ast\}\in\sigma(f)$.
\end{proof}

The next lemma plays a central r\^ole. Its proof is common in set
theory.
\begin{lemma}\label{lem:dens_Count}
For all $n\in\NN$, let $\pi_n$ be a discrete measurable equivalency.
Then there
exists a random variable $f$ such that
$\sigma(f)=\vee_n \sigma({X/\pi_n})$.
\end{lemma}
\begin{proof}
Before proving the core of the Lemma, we build a sequence
$(g_n)_{n\in\NN}$ of functions $g_n:\NN^n\to \R$ that will be used
to define the function $f$.

Take $h:\NN\cup\{0\}\to [0,1)$ to be the increasing function
$h(m)=1-2^{-m}$ and let $(g_n)_{n\in\NN}$ the sequence of function
$g_n:\NN^n\to \R$ so defined:
\begin{align*}
&g_1 (m_1) = h(m_1-1) \\
&g_2 (m_1,m_2) = g_1(m_1)+h(m_2-1)\Delta g_1(m_1) \\
&\vdots
\\
&g_{n+1} (\mathbf{m}_n,m_{n+1}) =
g_{n} (\mathbf{m}_n) +h(m_{n+1})\Delta g_{n} (\mathbf{m}_n) \\
&\vdots
\end{align*}
where, for all $n$, $\mathbf{m}_n=(m_1,\ldots,m_n)$ and
\[
\Delta g_{n} (\mathbf{m}_n) =
g_{n} (\mathbf{m}_{n-1},m_n+1)
-
g_{n} (\mathbf{m}_{n-1},m_n).
\]
As a first consequence of the definition,
note that for any choice of $n$ and
$\mathbf{m}_{n+1}$, it holds that
\begin{equation}\label{eq:prima_prop_g_n}
g_{n} (\mathbf{m}_{n-1},m_n)
\leq g_{n+1} (\mathbf{m}_{n+1}) 
< g_{n} (\mathbf{m}_{n-1},m_n+1)
\end{equation}
since $h\in[0,1)$.
We now prove by induction on $n_1+n_2$ that
for any choice of $n_1\in\NN$,
$n_2\in\NN\cup\{0\}$ and $\mathbf{m}_{n_1+n_2}$,
we have
\begin{equation}\label{eq:prop_g_n}
g_{n_1}(\mathbf m_{n_1-1},m_n) \leq
g_{n_1+n_2}(\mathbf{m}_{n_1+n_2}) <
g_{n_1}(\mathbf{m}_{n_1-1},m_n+1).
\end{equation}
Eq.~\eqref{eq:prop_g_n} is clearly true for $n_1+n_2=1$, since $h$
is strictly monotone. The same argument shows that
Eq.~\eqref{eq:prop_g_n} is always true for $n_2=0$ and therefore we
check it only for $n_2>0$. We assume by induction that
Eq.~\eqref{eq:prop_g_n} is true for $n_1+n_2\leq n$ and we prove it
for $n_1+n_2= n+1$. By using twice the induction hypothesis, as
$n_2-1\geq 0$, we obtain
\begin{align*}
g_{n_1}(\mathbf m_1,m_n) & \leq
g_{n_1+n_2-1}(\mathbf{m}_{n_1+n_2-2},m_{n_1+n_2-1}) \\
& <g_{n_1+n_2-1}(\mathbf{m}_{n_1+n_2-2},m_{n_1+n_2-1}+1)\\
&\leq g_{n_1}(\mathbf m_1,m_n+1).
\end{align*}
Eq.~\eqref{eq:prop_g_n} is now a consequence
of Eq.~\eqref{eq:prima_prop_g_n}.

\null

Now, we come back to the proof of the Lemma.
First note that, without loss of generality, we can (and we do)
require the sequence
$(\pi_n)_{n\in\NN}$ to be monotone, by
taking the sequence $\pi'_n = \cap_{i=1}^n \pi_i$
instead of $\pi_n$. $\pi'_n$ is again a countable measurable
equivalency on $X$. In fact, by Lemma~\ref{lem:discrete_measur}
we can read this statement in trivial terms of partitions:
since at most countable intersection
of families of countable measurable partition
is a countable measurable partition.
Moreover, by definition,
$\vee_{i=1}1^n \sigma(X/\pi_i)=\vee_{i=1}1^n \sigma(X/\pi'_i)$.
%

Let $\tau_n=X/\pi_n$ be the increasing sequence of countable
measurable dissections of $X$. We are going to give a consistent
inductive method of numbering the set of atoms of $\tau_n$ to build
the functions $f_n$. Let $\tau_1 = \{A_{1}^{(1)}, A_{2}^{(1)},
\ldots\}$ be any ordering of $\tau_1$. By induction, let
$\{A_{\mathbf{m}_n, 1}^{(n+1)}, A_{\mathbf{m}_n, 2}^{(n+1)},
\ldots\}$ be the partition of the atom
$A_{\mathbf{m}_n}^{(n)}\in\tau_n$ given by $\tau_{n+1}$. Define, for
any $n\in\NN$, \[f_n(x)=g_n(\mathbf{m}_n) \qquad \iff
\qquad x\in
A_{\mathbf{m}_n}^{(n)}.\]
To complete the proof, we first show that
$\sigma(f_n)=\sigma(X/\pi_n),\forall n$, and then we prove
$\sigma(f)=\sigma(f_1,f_2,\ldots)$ by proving that $f_n\to
f$ pointwise.

\null

To prove that $\sigma(f_n)=\sigma(X/\pi_n)$ we show that $f_n(x)=f_n(y)\iff
\exists\mathbf{m}_n\colon x,y\in A_{\mathbf{m}_n}^{(n)}$.
One implication is a consequence of the fact that $f_n$
is defined on the partition of $X$ given by $X/\pi_n=\tau_n$.
For the converse,
assume that
\(
x\in A_{\mathbf{m}_n}^{(n)} \neq A_{\mathbf{m}_n'}^{(n)}\ni y
\)
and consider $n_1:=\min\{j\leq n\colon m_j\neq m_j'\}$. Thus
$\mathbf{m}_{n_1-1}=\mathbf{m}_{n_1-1}'$ and,
without loss of generalities, $m_{n_1}<m'_{n_1}$.
By Eq.~\eqref{eq:prop_g_n}, we have
\begin{multline*}
f_n(x) = g_n(\mathbf{m}_n) \\
< g_{n_1} (\mathbf{m}_{n_1-1},m_{n_1}+1)
\leq g_{n_1} (\mathbf{m}_{n_1-1}',m_{n_1}') \leq g_n(\mathbf{m}_n')
=f_n(y).\\
\end{multline*}
$\sigma(f)=\sigma(f_1,f_2,\ldots) $.\\
$\subseteq$.
The sequence $(f_n)_n$ is monotone by definition and bounded
by Eq.~\eqref{eq:prop_g_n}. Then $\exists f\colon f_n\uparrow f$
and thus $\sigma(f)\subseteq \sigma(f_1,f_2,\ldots)$.

$\supseteq$. Let $n$ be fixed, and take $x,y\in X$ with
$f_n(x)<f_n(y)$. Then, for any $h\geq 0$,
$\tau_n\subseteq\tau_{n+h}$ implies \( x\in
A_{\mathbf{m}_{n+h}}^{(n+h)} \neq A_{\mathbf{m}_{n+h}'}^{(n+h)}\ni
y. \) As above, consider $n_1:=\min\{j\leq n\colon m_j\neq m_j'\}$.
As $f_n(x)<f_n(y)$, we have $\mathbf{m}_{n_1-1}=\mathbf{m}_{n_1-1}'$
and $m_{n_1}<m'_{n_1}$. Again, by Eq.~\eqref{eq:prop_g_n}, for
$h>n_1+1-n$,
\begin{align*}
f_{n+h}(x) & = g_{n+h}(\mathbf{m}_{n+h})
\\
& < g_{n_1+1}(\mathbf{m}_{n_1},m_{n_1+1}+1) = \alpha
\\
& < g_{n_1}(\mathbf{m}_{n_1-1},m_{n_1}+1)
\\
& \leq g_{n}(\mathbf{m}_{n}') = f_n(y) ,
\end{align*}
\emph{i.e.}, $\forall h$, $f_{n+h}(x)<\alpha<f_n(y)$. As $f_l\uparrow f$,
$f(x)< f(y)$.
Apply Lemma~\ref{lem:mon_fun_alg_1} with $g=f_n$ to
conclude that $\sigma(f_n)\subseteq\sigma(f)$.
\end{proof}

As a consequence of Lemma~\ref{lem:dens_Count}, any countably
generated sub $\sigma$-algebra is generated by a measurable
equivalence $\pi$, as the following lemma states.
\begin{lemma}\label{lem:count_gen_new}
$\mathcal A\subseteq\mathcal X$ is countably generated if and only
if there exists a random variable $f$ such that $\mathcal
A=\sigma(f)$.
\end{lemma}
\begin{proof}
$\Rightarrow$ Let $\mathcal A=\sigma(A_1,A_2,\ldots)$. Apply
Lemma~\ref{lem:dens_Count} with
$X/\pi_n=\{A_n,X\setminus A_n\}$.

$\Leftarrow$ Take a countable base $B_1,B_2,\ldots$ of $\mathcal
B_{\mathbb R}$ and simply note that
$\sigma(f)=\sigma(\{f^{-1}(B_1),f^{-1}(B_2),\ldots\})$.
\end{proof}

\begin{proof}[Proof of Lemma \ref{lem:ordering_1}]
Let $x\in X$ be fixed. By hypothesis, $\mathcal A_1 \subseteq
\mathcal A_2$. If $\mathcal A_1=\sigma(A^1_1, A^1_2,\ldots)$ then
$\mathcal A_2$ will be of the form  $\mathcal A_2=\sigma(A^1_1,
A^2_1,A^1_2,A^2_2,\ldots)$. Without loss of generality (if needed,
by choosing $X\setminus A^j_n$ instead of $A^j_n$) we can require
$x\in A^j_n$, for any $n\in\NN$ and $j=1,2$. Then $[x]_{\mathcal
A_2}=\cap_n (A^1_n\cap A^2_n)\subseteq \cap_n A^1_n =[x]_{\mathcal
A_1}$.

The last part of the proof is a consequence of
Lemma~\ref{lem:count_gen_new} and of the first point, since
\[
f^{-1}(\{f(x)\})=[x]_{\pi_f}\subseteq
[x]_{\pi_g}=g^{-1}(\{g(x)\}),
\]
or, equivalently, $f(x)=f(y)\Rightarrow g(x)=g(y)$
which is the thesis.
\end{proof}

\begin{proof}[Proof of Theorem~\ref{limit_eq}]
Note that $X/\pi_\infty\subseteq \mathcal X$ is countable, generated
by $\cup_n X/\pi_n$. Then $\pi_\infty$ is a measurable equivalency
by Lemma~\ref{lem:count_gen_new}.

Conversely, we can use the standard
approximation technique: if $\pi=\pi_f$ is measurable,
let $f_n = 2^{-n}\lfloor 2^n f\rfloor$ for any $n$.
Since $f_n$ are discrete random variables,
$\pi_n$ are defined through Lemma~\ref{lem:discrete_measur}.
By Lemma~\ref{lem:ordering_1} and Eq.~\eqref{eq:vee_and_cap},
the thesis $\pi_f=\cap_n\pi_n$
will be a consequence
of the fact that $\sigma(f)=\vee_n \sigma(f_n)$.

$\sigma(f_n)\subseteq \sigma(f)$ by definition,
which implies $\sigma(f_1,f_2,\ldots)\subseteq \sigma(f)$. Finally, as
$f_n\to f$, we have $\sigma(f)\subseteq \sigma(f_1,f_2,\ldots)$,
which completes the proof.
\end{proof}

\section{Proof of Theorem~\ref{thm:why_A0}}\label{app:A0}
Before proving the theorem, we state the following Lemma.
\begin{lemma}\label{lem:why_A02}
Let $(X,\mathcal X)$ be a measurable space.
\begin{enumerate}
 \item\label{col2:1} If any equivalence relationship $\pi$
on $X$ is measurable, then $\mathcal X=2^X$ and $\mathrm{card}(X)\leq
\mathrm{card}(\mathbb R)$.
 \item\label{col2:2} The converse is true under the axiom of choice.
\end{enumerate}
\end{lemma}
\begin{proof}
$\ref{col2:1}\Rightarrow \ref{col2:2}$.
Let $\pi_I$ be the identity relation: $ x\,\pi_I\,y \iff x=y$. By
hypothesis, there exists $f$ such that $\pi_I=\pi_f$, and thus $f$ is
injective. Then $\mathrm{card}(X)\leq \mathrm{card}(\mathbb R)$.
Now, take $A\subseteq X$ and let $\pi_A$ be the relation so defined:
\[
x\,\pi_A\,y \iff \{x,y\}\subseteq A\text{ or } \{x,y\}\subseteq
X\setminus A.
\]
Since any equivalency is measurable, then there exists
$f:(X,\mathcal X)\to(\mathbb R,\mathcal
B_\mathbb R)$ such that $\pi_A=\pi_f$. But $\sigma(f)=\sigma(A)$,
which shows that $A\subseteq X\Longrightarrow A\in\mathcal X$, i.e.\
$\mathcal X=2^X$.

$\ref{col2:2}\Rightarrow \ref{col2:1}$.
Since $\mathrm{card}(X)\leq \mathrm{card}(\mathbb R)$,
there exists an injective function $h:X\to\mathbb R$. Let $\pi$ be a
equivalence relationship on $X$, and define the following
equivalence on $\mathbb R$:
\[
 r_1\, R\, r_2 \iff \Big( \{r_1,r_2\}\subseteq h(X) \text{ and } h^{-1}(r_1)\,
\pi\, h^{-1}(r_2)\Big) \text{ or } \{r_1,r_2\}\subseteq \mathbb
R\setminus h(X)
\]
By definition of $R$, if we denote by $\pi_R$ the canonical
projection of $\mathbb R$ on $\mathbb R/R$, then $\pi_R\circ
h:X\to\mathbb R/R$ is such that
\[
\pi_R\circ h (x)=\pi_R\circ h (y) \iff x\, \pi\,y.
\]
The axiom of choice ensures the existence of a injective map
$g:\mathbb R/R\to \mathbb R$. Then $f:=g\circ \pi_R\circ h :X\to
\mathbb R$ is such that $\pi=\pi_f$. $f$ is measurable since
$\mathcal X = 2^X$.
\end{proof}

\begin{proof}[Proof of Theorem~\ref{thm:why_A0}]
$\ref{co:3}\Rightarrow \ref{co:2}$. By
Lemma~\ref{lem:why_A02} and  Assumption~\eqref{eq:A0def},
$(X,2^X)$ is weakly Blackwell.

$\ref{co:2}\Rightarrow \ref{co:1}$. Assume $X$ is uncountable. By CH,
exists $Y\subseteq X$ s.t.
$Y\mathop{\leftrightarrow}\limits^{g_1}\mathbb R$ (i.e.\ $Y$
is in bijection with $\mathbb R$ via $g_1$). Take a bijection
$\mathbb R\mathop{\leftrightarrow}\limits^{g_2}\mathbb R\setminus \{0\}$.
Then the map
\[
g(x) =
\begin{cases}
g_2(g_1(x)) & \text{if }x\in Y;\\
0 & \text{if }x\in X\setminus Y;
\end{cases}
\]
is a
bijective map from $\{Y,\{X\setminus Y\}\}$ to $\mathbb R$.
Equip $\mathbb R$ with the Borel $\sigma$-algebra $\mathcal
B_{\mathbb R}$ and let $\mathcal A_1=g^{-1}(\mathcal B_{\mathbb
R})$. $\mathcal A_1$ is countably generated and its atoms are all
the points in $Y$ and the set $X\setminus Y$. Now, take a non-Borel
set $N$ of the real line. $\mathcal A_2=g^{-1}( \sigma(\mathcal
B_{\mathbb R},N))$ is also countably generated, $\mathcal A_1\subsetneq
\mathcal A_2$  and its atoms are all
the points in $Y$ and the set $X\setminus Y$, too. Since $\mathcal
A_1\subseteq 2^X$ and $\mathcal A_2\subseteq 2^X$, $(X,2^X)$ is not
a weakly Blackwell space by Lemma~\ref{lem:A0_A1}.

$\ref{co:1}\Rightarrow \ref{co:3}$. Since $X$ is countable, then
$X/\pi$ is. Therefore, Lemma~\ref{lem:discrete_measur} ensures any equivalence
$\pi$ is measurable, since $\mathcal X = 2^X$.
Finally, just note that each countable set is strongly Blackwell.
And thus Lemma~\ref{lem:A0_A1} concludes the proof.
\end{proof}



\providecommand{\bysame}{\leavevmode\hbox to3em{\hrulefill}\thinspace}
\providecommand{\MR}{\relax\ifhmode\unskip\space\fi MR }
\providecommand{\MRhref}[2]{%
  \href{http://www.ams.org/mathscinet-getitem?mr=#1}{#2}
}
\providecommand{\href}[2]{#2}

\end{document}